\documentclass[submission,nomarks]{dmtcs-episciences}

\usepackage[utf8]{inputenc}
\usepackage{subfigure}
\usepackage{amsmath}

\newtheorem{thm}{Theorem}[section]
\newtheorem{conj}[thm]{Conjecture}
\newtheorem{cor}[thm]{Corollary}
\newtheorem{lemma}[thm]{Lemma}
\newtheorem{prop}[thm]{Proposition}
\newtheorem{fact}[thm]{Fact}
\newtheorem{defn}[thm]{Definition}
\newtheorem{remark}[thm]{Remark}

\def\t{\mathbf{t}}

\usepackage[round]{natbib}

\author{Marisa Gaetz}
\title[Anti-power $j$-fixes of the Thue-Morse word]{Anti-power $j$-fixes of the Thue-Morse word}
\affiliation{Massachusetts Institute of Technology}
\received{2019-05-20}
\revised{2020-08-05,2021-02-11}
\accepted{2021-02-18}

\begin{document}

\publicationdetails{23}{2021}{1}{5}{5483}
\maketitle

\begin{abstract}

Recently, Fici, Restivo, Silva, and Zamboni introduced the notion of a $k$-anti-power, which is defined as a word of the form $w^{(1)} w^{(2)} \cdots w^{(k)}$, where $w^{(1)}, w^{(2)}, \allowbreak \ldots, w^{(k)}$ are distinct words of the same length. For an infinite word $w$ and a positive integer $k$, define $AP_j(w,k)$ to be the set of all integers $m$ such that $w_{j+1} w_{j+2} \cdots w_{j+km}$ is a $k$-anti-power, where $w_i$ denotes the $i$-th letter of $w$. Define also $\mathcal{F}_j(k) = (2 \mathbb{Z}^+ - 1) \cap AP_j(\mathbf{t},k)$, where $\mathbf{t}$ denotes the Thue-Morse word. For all $k \in \mathbb{Z}^+$, $\gamma_j(k) = \min (AP_j(\mathbf{t},k))$ is a well-defined positive integer, and for $k \in \mathbb{Z}^+$ sufficiently large, $\Gamma_j(k) = \sup ((2 \mathbb{Z}^+ -1) \setminus \mathcal{F}_j(k))$ is a well-defined odd positive integer. In his 2018 paper, Defant shows that $\gamma_0(k)$ and $\Gamma_0(k)$ grow linearly in $k$. We generalize Defant's methods to prove that $\gamma_j(k)$ and $\Gamma_j(k)$ grow linearly in $k$ for any nonnegative integer $j$. In particular, we show that $\displaystyle 1/10 \leq \liminf_{k \rightarrow \infty} (\gamma_j(k)/k) \leq 9/10$ and $\displaystyle 1/5 \leq \limsup_{k \rightarrow \infty} (\gamma_j(k)/k) \leq 3/2$. Additionally, we show that $\displaystyle \liminf_{k \rightarrow \infty} (\Gamma_j(k)/k) = 3/2$ and $\displaystyle \limsup_{k \rightarrow \infty} (\Gamma_j(k)/k) = 3$.
\end{abstract}

\section{Introduction}

A finite word is called a \textit{$k$-power} if it is of the form $w^k$ for some word $w$. A particularly famous consequence of the study of $k$-powers is Axel Thue's 1912 paper \cite{Thue}, which introduces an infinite binary word that does not contain any 3-powers as subwords. This word has since caught the interest of numerous academicians \cite{Cohen, Shallit, Brlek, Han, Cooper, Defant, Dejean, Mahler, Narayanan, Palacios} spanning the fields of combinatorics, analytic number theory \cite{Cohen}, game theory \cite{Cooper}, and economics \cite{Palacios}. It is now known as the Thue-Morse word.

\begin{defn} \label{def:Thue}
Let $A_0 = 0$. For each nonnegative integer $n$, let $B_n = \overline{A_n}$ be the Boolean complement of $A_n$, and let $A_{n+1} = A_n B_n$. The \emph{Thue-Morse word} $\t$ is defined as
\begin{equation*}
\t = \lim_{n \rightarrow \infty} A_n = 0110100110010110 \cdots .
\end{equation*}
\end{defn}

As a natural adaptation of the Ramsey-type notion of a $k$-power, Fici, Restivo, Silva, and Zamboni \cite{Fici} introduce the anti-Ramsey-type notion of a $k$-anti-power. A \textit{$k$-anti-power} is a word $w$ of the form $w = w^{(1)} w^{(2)} \cdots w^{(k)}$, where $w^{(1)}, w^{(2)}, \ldots, w^{(k)}$ are distinct words of the same length. For example, 110100 is a 3-anti-power, while 101011 is not. Since the introduction of this notion in 2016, $k$-anti-powers have received much attention \cite{Badkobeh,Burcroff,Defant,Narayanan}. 

As their main result, Fici et al.~show that every infinite word contains powers of any order or anti-powers of any order. In doing so, they define the following set, which corresponds to an infinite word $w$ and a positive integer $k$:
\begin{equation*}
AP(w,k) = \{ m \in \mathbb{Z}^+ \; \vert \; w_1 w_2 \cdots w_{km} \text{ is a } k \text{-anti-power}  \}.
\end{equation*}
Here, $w_i$ indicates the $i$-th letter of the infinite word $w$. Such subwords (i.e. those starting from the first index of $w$) are called \textit{prefixes} of $w$. In \cite{Defant}, Defant introduces the generalized definition
\begin{equation*}
AP_j(w,k) = \{ m \in \mathbb{Z}^+ \; \vert \; w_{j+1}w_{j+2} \cdots w_{j+km} \text{ is a } k \text{-anti-power}  \},
\end{equation*}
himself studying $AP_0(\t,k) = AP(\t,k)$. Subwords beginning at the $(j+1)$-st index of a word $w$ will be referred to as \textit{$j$-fixes} of $w$. An easy consequence of \cite[Theorem 6]{Fici} is that $AP_j(\t,k)$ is nonempty for any nonnegative integer $j$ and all positive integers $k$. Therefore, we can make the following definition:
\begin{defn} \label{def:gamma}
Let $\gamma_j(k) = \min (AP_j(\t,k))$.
\end{defn}
For $j = 0$, it is the case that $m \in AP_0(\t,k)$ if and only if $2m \in AP_0(\t,k)$ (see Remark \ref{2m}). As a consequence, the only interesting elements of $AP_0(\t,k)$ are those that are odd. Thus, Defant \cite{Defant} makes the following definition for $j=0$ (which we have written in terms of arbitrary $j \in \mathbb{Z}^{\geq 0}$):
\begin{defn} \label{def:Gamma}
Let $\mathcal{F}_j(k)$ denote the set of odd positive integers $m$ such that the $j$-fix of $\t$ of length $km$ is a $k$-anti-power. Let $\Gamma_j(k) = \sup ((2 \mathbb{Z}^+ -1 ) \setminus \mathcal{F}_j(k))$. 
\end{defn}

For sufficiently large $k$, $\Gamma_j(k)$ is a well-defined odd positive integer (see Remark \ref{nonempty}). However, if $j \neq 0$, it is not necessarily the case that $m \in AP_j(\t,k)$ if and only if $2m \in AP_j(\t,k)$. For example, $4 \in AP_2(\t, 3)$, whereas $2 \not \in AP_2(\t, 3)$. However, we will later prove that the statement ``$m \in AP_j(\t,k)$ if and only if $2m \in AP_j(\t,k)$" holds for sufficiently large $m$ (see Corollary \ref{conj}), so it still makes sense to define $\Gamma_j(\t,k)$ in this way. See Section \ref{sec:Gamma} for further motivation for this definition.

\begin{remark} \label{nondecreasing}
It is immediate from Definition \ref{def:Gamma} that $\mathcal{F}_j(1) \supseteq \mathcal{F}_j(2) \supseteq \mathcal{F}_j(3) \supseteq \cdots$ for any $j \in \mathbb{Z}^{\geq 0}$. It follows that $\gamma_j(1) \leq \gamma_j(2) \leq \gamma_j(3) \leq \cdots$ and that $\Gamma_j(k)$ is nondecreasing when it is finite.
\end{remark}

As a means to understanding $\gamma_j(k)$ and $\Gamma_j(k)$, it will often be useful to consider the following related function:
\begin{defn} \label{def:K}
For a positive integer $m$, let $\mathfrak{K}_j(m)$ denote the smallest positive integer $k$ such that the $j$-fix of $\t$ of length $km$ is not a $k$-anti-power.
\end{defn}
A simple application of the Pigeonhole Principle gives that $\mathfrak{K}_j(m) \leq 2^m +1$. However, Defant \cite{Defant} and Narayanan \cite{Narayanan} prove significantly better bounds on $\mathfrak{K}_0(m)$, showing it grows linearly in $m$. Using these bounds, Defant \cite{Defant} is ultimately able to show the following:

\begin{thm}[\cite{Defant}]
\hspace{1cm}
\begin{itemize}
\item $\displaystyle \frac{1}{4}$\footnote{Erroneously stated in \cite{Defant} as $1/2$ (as will later be explained)}$\displaystyle \leq \liminf_{k \rightarrow \infty} \frac{\gamma_0(k)}{k} \leq \frac{9}{10}$
\item $\displaystyle \frac{1}{2}$\footnote{Erroneously stated in \cite{Defant} as $1$ (as will later be explained)}$\displaystyle \leq \limsup_{k \rightarrow \infty} \frac{\gamma_0(k)}{k} \leq \frac{3}{2}$
\item $\displaystyle \liminf_{k \rightarrow \infty} \frac{\Gamma_0(k)}{k} = \frac{3}{2}$
\item $\displaystyle \limsup_{k \rightarrow \infty} \frac{\Gamma_0(k)}{k} = 3$.
\end{itemize}
\end{thm}
Narayanan \cite{Narayanan} improves the above asymptotic bounds in the following way:
\begin{thm}[\cite{Narayanan}]
\hspace{1cm}
\begin{itemize}
\item $\displaystyle \frac{3}{4} \leq \liminf_{k \rightarrow \infty} \frac{\gamma_0(k)}{k} \leq \frac{9}{10}$
\item $\displaystyle \limsup_{k \rightarrow \infty} \frac{\gamma_0(k)}{k} = \frac{3}{2}$.
\end{itemize}
\end{thm}
The goal of this paper is to demonstrate similarly good bounds on the asymptotic growth of $\gamma_j(k)$ and $\Gamma_j(k)$ for general $j$. To do so, we will roughly follow the outline of Defant's paper \cite{Defant}, generalizing his bounds for $\mathfrak{K}_0(m)$ to bounds for $\mathfrak{K}_j(m)$; this will in turn allow us to prove that $\gamma_j(k)$ and $\Gamma_j(k)$ grow linearly in $k$. Specifically, we aim to prove the following:
\begin{itemize}
\item $\displaystyle \frac{1}{10} \leq \liminf_{k \rightarrow \infty} \frac{\gamma_j(k)}{k} \leq \frac{9}{10}$
\item $\displaystyle \frac{1}{5} \leq \limsup_{k \rightarrow \infty} \frac{\gamma_j(k)}{k} \leq \frac{3}{2}$
\item $\displaystyle \liminf_{k \rightarrow \infty} \frac{\Gamma_j(k)}{k} = \frac{3}{2}$
\item $\displaystyle \limsup_{k \rightarrow \infty} \frac{\Gamma_j(k)}{k} = 3$.
\end{itemize}

\begin{remark}
Note that we follow the methods of Defant \cite{Defant} rather than those of Narayanan \cite{Narayanan}, which seem more difficult to generalize to arbitrary $j \in \mathbb{Z}^{\geq 0}$.
\end{remark}

In Section \ref{sec:Thue}, we cover preliminary results relating to the Thue-Morse word. In Section \ref{sec:gamma} (resp. Section \ref{sec:Gamma}), we prove the aforementioned asymptotic bounds on $\gamma_j(k)/k$ (resp. $\Gamma_j(k)/k$).

\section{Properties of the Thue-Morse Word} \label{sec:Thue}

In this section, we will discuss some properties of the Thue-Morse word $\t = \t_1 \t_2 \t_3 \cdots$ that will be of use throughout the remainder of the paper. It is well known that the $i$-th letter $\t_i$ of the Thue-Morse word has the same parity as the number of 1's in the binary expansion of $i-1$. In his 1912 paper \cite{Thue}, Thue proved that $\t$ is \textit{overlap-free}, meaning that if $x$ and $y$ are finite words (with $x$ nonempty), then $\t$ does not contain $xyxyx$ as a subword. Taking $y$ to be empty shows that $\t$ does not contain any 3-powers as subwords. 

Let $\mathcal{W}_1$ and $\mathcal{W}_2$ be sets of words. We say a function $f: \mathcal{W}_1 \rightarrow \mathcal{W}_2$ is a \textit{morphism} if $f(xy) = f(x) f(y)$ for all words $x,y \in \mathcal{W}_1$. We will write $\mathbb{A}^{\leq \omega}$ to refer to the set of all words over an alphabet $\mathbb{A}$. Using this notation, let $\mu : \{ 0,1 \}^{\leq \omega} \rightarrow \{ 01,10 \}^{\leq \omega}$ be the morphism uniquely defined by $\mu (0) = 01$ and $\mu (1) = 10$. Similarly, let $\sigma : \{ 01,10 \}^{\leq \omega} \rightarrow \{ 0,1 \}^{\leq \omega}$ be the morphism uniquely defined by $\sigma (01) = 0$ and $\sigma (10) = 1$. The Thue-Morse word $\t$ and its Boolean complement $\overline{\t}$ are the unique one-sided infinite words over the alphabet $\{ 0,1 \}$ that are fixed by $\mu$. Similarly, $\t$ and $\overline{\t}$, as viewed over the alphabet $\{ 01,10 \}$, are the unique one-sided infinite words fixed by $\sigma$. The observation that $\mu (\t) = \t$ allows us to view $\t$ as a word over the alphabet $\{ 01,10 \}$. More generally, if we recall the definitions of $A_n$ and $B_n$ from Definition \ref{def:Thue} and note the equalities $A_n = \mu^n(0)$ and $B_n = \mu^n(1)$, we can view $\t$ as a word over the alphabet $\{ A_n, B_n \}$.

\begin{remark} \label{2m}
Using that $\mu (\t) = \t$ and $\sigma(\t) = \t$, it is straightforward to see that $m \in AP_0(\t,k)$ if and only if $2m \in AP_0(\t,k)$.
\end{remark}

We will follow Defant \cite{Defant} in using the notation $\langle \alpha , \beta \rangle = \t_{\alpha} \t_{\alpha+1} \cdots \t_{\beta}$ for any positive integers $\alpha, \beta$ with $\alpha \leq \beta$. We are now in a position to establish some preliminary results relating to $\t$.

\begin{fact}
For any positive integers $n$ and $r$, $\langle 2^n r +1, 2^n(r+1) \rangle = \mu^n (\t_{r+1})$.
\end{fact}

\begin{lemma} \label{lem:properties}
For $m \in \mathbb{Z}^+$, $\t_{2m+1} \neq \t_{2m+2}$.
\end{lemma}

\begin{proof}
If $\t_{m+1} = 1$, then $\mu(\t_{m+1}) = \t_{2m+1} \t_{2m+2} = 10$. Similarly, if $\t_{m+1} = 0$, then $\mu(\t_{m+1}) = \t_{2m+1} \t_{2m+2} = 01$. In either case, $\t_{2m+1} \neq \t_{2m+2}$. 
\end{proof}

\begin{lemma} \label{lem:even-aid}
Let $k \in \mathbb{Z}^+$. Then $\t_{2k +1} \t_{2k+2} = \t_{4k+1} \t_{4k+2}$.
\end{lemma}

\begin{proof}
Fix some $k \in \mathbb{Z}^+$ and suppose that $\t_{k+1} = 1$. (The case in which $\t_{k+1}=0$ can be done similarly.) Note that $\mu (\t_{k+1}) = \t_{2k+1} \t_{2k+2} = 10$. Similarly, $\mu (\t_{2k+1}) = \t_{4k+1} \t_{4k+2} = 10$. So we have that $\t_{2k+1} \t_{2k+2} = 10 = \t_{4k+1} \t_{4k+2}$, as desired. 
\end{proof}  

\section{Asymptotics for $\gamma_j(k)$} \label{sec:gamma}

In this section, we prove that $\displaystyle \frac{1}{10} \leq \liminf_{k \rightarrow \infty} \frac{\gamma_j(k)}{k} \leq \frac{9}{10}$ and $\displaystyle \frac{1}{5} \leq \limsup_{k \rightarrow \infty} \frac{\gamma_j(k)}{k} \leq \frac{3}{2}$. These asymptotic results relate to a conjecture from a recent paper of Berger and Defant \cite{ColinAaron}, which states that for any sufficiently well-behaved \textit{aperiodic}, \textit{morphic} word $W$ (see \cite{ColinAaron} for the relevant definitions), there exists a constant $C=C(W)$ such that for all integers $j\geq 0$ and $k\geq 1$, $W$ contains a $k$-anti-power $j$-fix of length at most $C k^2$. The lower bounds proved in this section indicate that the quadratic upper bound in this conjecture is the best that could be true in general. Our upper bounds obtained in the next section also prove an effective version of this conjecture in the case of $W = \mathbf{t}$ (which Berger and Defant do not obtain).

Many of the statements proved in this section for general $j$ have $j = 0$ analogues in \cite{Defant}. Moreover, many of the proofs of these general statements closely follow the corresponding proofs in \cite{Defant}. In these cases, we highlight the key differences between the corresponding proofs and refer the reader to \cite{Defant} for the remaining details. 

\subsection{Lower Bounds for $\gamma_j(k)/k$}

In this subsection, we present a series of lemmas that collectively establish an upper bound for $\mathfrak{K}_j(m)$ for any integer $m \geq 2$. This will allow us to establish lower bounds for $\displaystyle \liminf_{k \rightarrow \infty} (\gamma_j(k)/k)$ and $\displaystyle \limsup_{k \rightarrow \infty} (\gamma_j(k)/k)$. We begin with three lemmas that we will apply in the proofs of many of the lemmas later in this subsection.

\begin{lemma} \label{lem-aid}
Let $m,j \in \mathbb{Z}^{\geq 0}$ with $m \geq 2$, and let $\ell = \left \lceil \log_2(m+j) \right \rceil$. For any $s, a \in \mathbb{Z}^{+}$, there exists a nonnegative integer $r$ such that
\begin{equation*}
2^{\ell}(s-1)+1 \leq rm+j+1 < (r+1)m+j < 2^{\ell}(s+a).
\end{equation*}
\end{lemma}

\begin{proof}
Fix some $s,a \in \mathbb{Z}^{+}$. Note that 
\begin{equation}
    2^{\ell}(s+a) - 2^{\ell}(s-1) = 2^{\ell}(a+1) \geq 2^{\ell +1} \geq 2(m+j) \geq 2m.
\end{equation}
Since $(r+1)m+j-(rm+j) = m$ for any integer $r$, it follows that there exists $r \in \mathbb{Z}$ satisfying
\begin{equation} \label{r-ineq}
2^{\ell}(s-1)+1 \leq rm+j+1 < (r+1)m+j < 2^{\ell}(s+a).
\end{equation}
Moreover, we can always choose $r$ to be nonnegative; to verify this fact, it suffices to check that $r=0$ satisfies (\ref{r-ineq}) when $s = 1$:
\begin{equation}
2^{\ell}(s-1)+1 = 1 \leq j + 1 < m + j < 2^{\ell +1} \leq 2^{\ell}(s+a).
\end{equation}
When $s \geq 2$, any integer $r$ satisfying (\ref{r-ineq}) is clearly positive.  
\end{proof}

\begin{lemma}[{cf.~\cite[Lemma 12]{Defant}}] \label{gen4.1}
Let $j \in \mathbb{Z}^{\geq 0}$, $m \in \mathbb{Z}^+$, and $\ell = \left \lceil \log_2(m+j) \right \rceil$. If $\mathfrak{K}_j(m) > 2^{\ell} + 1$, then $\mathbf{t}_{m+1} \mathbf{t}_{m+2} = 11$ and $\mathbf{t}_{2m+1} \mathbf{t}_{2m+2} = 10$.
\end{lemma}

\begin{proof}
Suppose $\mathfrak{K}_j(m) > 2^{\ell} + 1$. Let $w_0 = \langle j+1, m+j \rangle$, $w_1 = \langle 2^{\ell -1}m+j+1, (2^{\ell - 1} +1)m +j \rangle$, and $w_2 = \langle 2^{\ell}m+j+1, (2^{\ell}+1)m+j \rangle$. By our assumption that $\mathfrak{K}_j(m) > 2^{\ell} + 1$, we have that $w_0$, $w_1$, and $w_2$ are distinct. Notice that for each $n \in \{ 0,1,2 \}$, the word $w_n$ is a $j$-fix of 
\begin{equation}
\langle nm2^{\ell -1} + 1, (nm+2) 2^{\ell -1} \rangle = \mu^{\ell -1}(\mathbf{t}_{nm+1} \mathbf{t}_{nm+2}).
\end{equation}
It follows that $\mathbf{t}_1 \mathbf{t}_2$, $\mathbf{t}_{m+1} \mathbf{t}_{m+2}$, and $\mathbf{t}_{2m+1} \mathbf{t}_{2m+2}$ are distinct. Note that $\mathbf{t}_1 \mathbf{t}_2 = 01$ and that $\mathbf{t}_{2m+1} \neq \mathbf{t}_{2m+2}$ (by Lemma \ref{lem:properties}); hence, $\mathbf{t}_{2m+1} \mathbf{t}_{2m+2} = 10$. Therefore, $\mu (\t_{m+1}) = \t_{2m+1} \t_{2m+2} = 10$, which implies that $\t_{m+1} = 1$. Consequently, $\mathbf{t}_{m+1} \mathbf{t}_{m+2} = 11$.
\end{proof}

\begin{lemma}[{cf.~\cite[Lemma 13]{Defant}}] \label{gen4.2}
Let $j, m \in \mathbb{Z}^{\geq 0}$ with $m \geq 2$, and let $\ell = \left \lceil \log_2(m+j) \right \rceil$. Suppose there exists $s \in \mathbb{Z}^+$ such that $\mathbf{t}_{s} \mathbf{t}_{s+1} = \mathbf{t}_{m+s} \mathbf{t}_{m+s+1}$. Then 
\begin{equation*}
\mathfrak{K}_j(m) < 2^{\ell} + \frac{2^{\ell}(s+1)-j}{m}.
\end{equation*}
\end{lemma}

\begin{proof}
This proof follows from a straightforward modification of the proof of \cite[Lemma 13]{Defant} and an application of Lemma \ref{lem-aid} with $a = 1$. 
\end{proof}

Now that we have established the preceding preliminary results, we are ready to derive upper bounds for $\mathfrak{K}_j(m)$ for all integers $m \geq 2$. We consider the cases $m \equiv 0 \pmod 2$, $m \equiv 1 \pmod 8$, $m \equiv 29 \pmod{32}$, and remaining values of $m$. We then combine the bounds derived in each of these cases into a uniform upper bound on $\mathfrak{K}_j(m)$. We first consider the case in which $m \equiv 0 \pmod 2$.

\begin{lemma} \label{even}
Let $m = 2^L k$, where $L,k \in \mathbb{Z}^{+}$. Let $j \in \mathbb{Z}^{\geq 0}$, and let $\ell = \left \lceil \log_2(m+j) \right \rceil$. Then
\begin{equation*}
\mathfrak{K}_j(m) < 2^{\ell +1} + \frac{2^{\ell +1} -j}{m}.
\end{equation*}
\end{lemma}

\begin{proof}
By Lemma \ref{lem:even-aid}, we have that $\t_{2^Lk +1} \t_{2^L k+2} = \t_{2^{L+1}k+1} \t_{2^{L+1}k+2}$. It follows that 
\begin{equation}
\langle 2^{\ell}m+1, 2^{\ell}(m+2) \rangle = \mu^{\ell}(\mathbf{t}_{m+1} \mathbf{t}_{m+2}) = \mu^{\ell}(\mathbf{t}_{2m+1} \mathbf{t}_{2m+2}) = \langle 2^{\ell +1} m  +1, 2^{\ell +1}(m+1) \rangle.
\end{equation}
Applying Lemma \ref{lem-aid} with $s = 1$ and $a = 1$ shows that there exists $r \in \mathbb{Z}^{\geq 0}$ such that 
\begin{equation} \label{eq:3}
    1 \leq rm+j+1 < (r+1)m +j < 2^{\ell + 1}.
\end{equation}
We can thus write the prefix of $\t$ of length $2^{\ell +1}$ as
\begin{equation}
\langle 1, 2^{\ell +1} \rangle = w \langle rm + j + 1, (r+1)m+j \rangle z,
\end{equation}
where $w = \langle 1 , rm+j \rangle$ and $z = \langle (r+1)m+j+1, 2^{\ell+1} \rangle$. Adding $2^{\ell} m$ everywhere in (\ref{eq:3}) similarly shows that we can write
\begin{equation} \label{yee1}
\langle 2^{\ell} m +1, 2^{\ell}(m+2) \rangle = w' \langle (2^{\ell}+r)m+j+1, (2^{\ell}+r+1)m+j \rangle z',
\end{equation}
where $w' = \langle 2^{\ell}m+1, (2^{\ell}+r)m+j \rangle$ and $z' = \langle (2^{\ell}+r+1)m+j+1, 2^{\ell}(m+2) \rangle$. In the same way, adding $2^{\ell +1}m$ everywhere in (\ref{eq:3}) gives that 
\begin{equation} \label{yee2}
\langle 2^{\ell +1}m +1, 2^{\ell +1}(m+1) \rangle = w'' \langle (2^{\ell +1}+r)m+j+1, (2^{\ell +1}+r+1)m +j \rangle  z'',
\end{equation}
where $w'' = \langle 2^{\ell+1}m +1, (2^{\ell+1}+r)m+j \rangle$ and $z'' = \langle (2^{\ell+1}+r+1)m+j+1, 2^{\ell+1}(m+1) \rangle$. Observe that $\vert w'' \vert = rm+j = \vert w' \vert$. As a result, Equations (\ref{yee1}) and (\ref{yee2}) give that
\begin{equation}
\langle (2^{\ell}+r)m+j+1, (2^{\ell}+r+1)m+j \rangle = \langle (2^{\ell +1}+r)m+j+1, (2^{\ell +1}+r+1)m+j \rangle.
\end{equation}
Using (\ref{eq:3}) to note that $\displaystyle r+1 < \frac{2^{\ell +1} -j}{m}$, we get
\begin{equation}
\mathfrak{K}_j(m) \leq  2^{\ell +1}+r +1 < 2^{\ell +1} + \frac{2^{\ell+1}-j}{m},
\end{equation}
as desired. 
\end{proof}

The following two lemmas establish upper bounds for $\mathfrak{K}_j(m)$ when $m \equiv 1 \pmod 8$. Setting $j = 0$ in Lemma \ref{gen4.4} implies Defant's result \cite[Lemma 15]{Defant}, while setting $j = 0$ in Lemma \ref{gen4.5} gives a bound for $\mathfrak{K}_0(m)$ that is worse than the one given in \cite[Lemma 16]{Defant} by a factor of two. 

\begin{lemma}[{cf.~\cite[Lemma 15]{Defant}}] \label{gen4.4}
Let $j \in \mathbb{Z}^{\geq 0}$, and suppose $m = 2^L h +1$, where $L$ and $h$ are integers with $L \geq 3$ and $h$ odd. Let $\ell = \left \lceil \log_2(m+j) \right \rceil$. We have
\begin{equation*}
\mathfrak{K}_j(m) < 2^{\ell} + \frac{2^{\ell}(2^{L+1}+4)-j}{m}.
\end{equation*}
\end{lemma}

\begin{proof}
Suppose, for the sake of contradiction, that $\displaystyle \mathfrak{K}_j(m) \geq 2^{\ell} + \frac{2^{\ell}(2^{L+1}+4)-j}{m}$.
We obtain a contradiction to Lemma \ref{gen4.2} by finding a positive integer $s \leq 2^{L+1}+3$ such that $\mathbf{t}_{s} \mathbf{t}_{s+1} = \mathbf{t}_{m+s} \mathbf{t}_{m+s+1}$. Note that $m$ has a binary expansion of the form $x01^{r}0^{L-1}1$, where $x$ is a (possibly empty) binary string. Since $m \geq 2^3 \cdot 1 + 1 = 9$, we have that $r \geq 1$. Let $N$ be the number of 1's in $x$. The binary expansion of $m+2^L+2$ can be expressed as $x10^{r+L-2}11$, which has $N+3$ 1's. Similarly, we obtain the following table:
\begin{center}
\begin{tabular}{c | c | c}
$i$ & Binary Expansion of $i$ & Number of 1's in Binary Expansion of $i$ \\ \hline
$m+2^L+2$ & $x10^{r+L-2}11$ & $N + 3$ \\
$m+2^L +3$ & $x10^{r+L-3}100$ & $N + 2$ \\
$m+2^{L+1} +2$ & $x10^{r-1}10^{L-2}11$ & $N + 4$ \\
$m+2^{L+1}+3$ & $x10^{r-1}10^{L-3}100$ & $N+3$
\end{tabular}
\end{center}
Recall that the parity of $\mathbf{t}_i$ is the same as the parity of the number of 1's in the binary expansion of $i-1$. It follows that $\mathbf{t}_{m+2^L+3} \mathbf{t}_{m+2^L+4} = 01$ if $N$ is odd and $\mathbf{t}_{m+2^{L+1}+3} \mathbf{t}_{m+2^{L+1}+4} = 01$ if $N$ is even. Observe that $\mathbf{t}_{2^L+3} \mathbf{t}_{2^L+4} = \mathbf{t}_{2^{L+1}+3} \mathbf{t}_{2^{L+1}+4} = 01$. Therefore, setting $s = 2^L+3$ yields a contradiction to Lemma \ref{gen4.2} if $N$ is odd, and setting $s = 2^{L+1}+3$ yields the desired contradiction if $N$ is even. 
\end{proof}

\begin{remark}
The proof of Lemma \ref{gen4.4} closely follows that of \cite[Lemma 15]{Defant}. Note, however, that in Defant's proof of \cite[Lemma 15]{Defant}, he mistakenly claims that $\t_{2^L+3} \t_{2^L + 4} = \t_{2^{L+1}+3} \t_{2^{L+1}+4} = 10$, rather than $\t_{2^L+3} \t_{2^L + 4} = \t_{2^{L+1}+3} \t_{2^{L+1}+4} = 01$. Setting $j = 0$ in the above proof yields a correct proof of \cite[Lemma 15]{Defant}.
\end{remark}

\begin{lemma}[{cf.~\cite[Lemma 16]{Defant}}] \label{gen4.5} 
Let $j \in \mathbb{Z}^{\geq 0}$. Suppose $m = 2^L h+1$, where $L$ and $h$ are integers with $L \geq 3$ and $h$ odd. Let $\ell = \left \lceil \log_2(m+j) \right \rceil$. If $n$ is an integer such that $2 \leq n \leq 2^{L-1}$, $\mathbf{t}_{m-n} = \mathbf{t}_{m-n+1}$, and $m+j \leq \left ( 1 - \frac{1}{2n+2} \right ) 2^{\ell}$, then
\begin{equation*}
\mathfrak{K}_j(m) \leq 2^{\ell+1} - \frac{2^{\ell+1}(n-\frac{1}{2}) + j}{m}.
\end{equation*}
\end{lemma}

\begin{proof}
By the proof of \cite[Lemma 16]{Defant}, we have that $\mathbf{t}_{m-2n}\mathbf{t}_{m-2n+1} = \mathbf{t}_{2m-2n} \mathbf{t}_{2m-2n+1}$ for any $m$ and $n$ satisfying the hypotheses of the lemma. Consequently,
\begin{align}
\langle (m-2n -1) 2^{\ell} +1, (m-2n+1)2^{\ell} \rangle &= \mu^{\ell}(\t_{m-2n} \t_{m-2n+1}) \\
&= \mu^{\ell}(\t_{2m-2n} \t_{2m-2n+1}) \\
&= \langle (2m-2n-1)2^{\ell} +1, (2m-2n+1)2^{\ell} \rangle.
\end{align}
We want to show that there is an integer $r \leq 2^{\ell}-1$ such that  
\begin{equation} \label{ineq2}
(m-2n-1)2^{\ell}+1 \leq (2^{\ell}-r-1)m+j+1 < (2^{\ell}-r)m + j < (m-2n+1)2^{\ell}.
\end{equation}
To this end, note that
\begin{equation}
(m-2n+1)2^{\ell} - (m-2n-1)2^{\ell} = 2 \cdot 2^{\ell} \geq 2(m+j) \geq 2m
\end{equation}
and that
\begin{equation}
((2^{\ell}-r)m+j)- ((2^{\ell}-r-1)m+j) = m.
\end{equation}
It follows that there exists $r \in \mathbb{Z}$ satisfying (\ref{ineq2}). To see that $r$ can always be chosen such that $r \leq 2^{\ell} - 1$, it suffices to note that our choice of $r$ is forced to be largest when $n$ is maximal (i.e.~when $n=2^{L-1}$), and that (\ref{ineq2}) is satisfied by $r = 2^{\ell} -1$ in this case. Therefore, for some integer $r \leq 2^{\ell} -1$, we have
\begin{equation} \label{eq:7} 
\langle (m-2n-1)2^{\ell} +1, (m-2n+1)2^{\ell} \rangle = w \langle (2^{\ell}-r-1)m+j+1, (2^{\ell}-r)m+j \rangle z,
\end{equation}
where $w = \langle (m-2n-1)2^{\ell} +1, (2^{\ell}-r-1)m+j \rangle$ and $z = \langle (2^{\ell}-r)m+j+1, (m-2n+1)2^{\ell} \rangle$. Adding $2^{\ell}m$ everywhere in (\ref{ineq2}) similarly gives that 
\begin{equation} \label{eq:8}
\langle (2m-2n-1)2^{\ell} +1, (2m-2n+1)2^{\ell} \rangle = w' \langle (2^{\ell+1}-r-1)m+j+1, (2^{\ell+1}-r)m+j \rangle z',
\end{equation}
where $w' = \langle (2m-2n-1)2^{\ell}+1, (2^{\ell +1}-r-1)m+j \rangle$ and $z' = \langle (2^{\ell +1}-r)m +j+1, (2m-2n+1)2^{\ell} \rangle$. Note that $\vert w' \vert = -rm-m+j+2^{\ell+1}m + 2^{\ell} = \vert w \vert$. Therefore, (\ref{eq:7}) and (\ref{eq:8}) give that 
\begin{equation}
\langle (2^{\ell}-r-1)m+j+1, (2^{\ell}-r)m+j \rangle = \langle (2^{\ell +1} -r-1)m+j+1, (2^{\ell+1}-r)m+j \rangle.
\end{equation}
Noting from (\ref{ineq2}) that $\displaystyle r > \frac{2^{\ell +1}(n-\frac{1}{2})+j}{m}$, we have
\begin{equation}
\mathfrak{K}_j(m) \leq 2^{\ell+1} - r \leq 2^{\ell +1} - \frac{2^{\ell+1}(n-\frac{1}{2})+j}{m},
\end{equation}
as desired.
\end{proof}

\begin{remark} \label{rmk:error2defant}
We make note of an error in Defant's proof of \cite[Lemma 16]{Defant}. In his proof, Defant claims that for $m$, $n$, and $\ell$ satisfying his hypotheses, 
\begin{equation} \label{error2:defant}
    \left ( 2^{\ell-1} , \left ( 1 - \frac{1}{2n+2} \right ) 2^{\ell} \right ] \subseteq \bigcup_{r=n}^{2n-1} \left [ \frac{2n-2}{r} 2^{\ell - 1}, \frac{2n+1}{r+1} 2^{\ell - 1} \right ].
\end{equation}
However, the intervals in the RHS of (\ref{error2:defant}) do not always overlap, so we see that (\ref{error2:defant}) is in fact false. Fortunately, setting $j = 0$ in Lemma \ref{gen4.5} gives the bound $\displaystyle \mathfrak{K}_0(m) \leq 2^{\ell +1} - \frac{2^{\ell + 1} (n-\frac{1}{2})}{m}$, which is only slightly worse than Defant's intended bound of $\mathfrak{K}_0(m) \leq 2^{\ell} -n$. This worsens Defant's lower bound for $\displaystyle \liminf_{k \rightarrow \infty} (\gamma_0(k)/k)$ from $1/2$ to $1/4$, and his lower bound for $\displaystyle \limsup_{k \rightarrow \infty} (\gamma_0(k)/k)$ from $1$ to $1/2$. However, Narayanan \cite{Narayanan} proves $\liminf_{k \rightarrow \infty} (\gamma_0(k)/k) \geq 3/4$ and $\limsup_{k \rightarrow \infty} (\gamma_0(k)/k) = 3/2$, so we still know Defant's claimed lower bounds to be true.
\end{remark}

We now address the case in which $m \equiv 29 \pmod{32}$. 

\begin{lemma}[{cf.~\cite[Lemma 14]{Defant}}] \label{29mod32} 
Let $m$ be a positive integer satisfying $m \equiv 29 \pmod{32}$. Let $j \in \mathbb{Z}^{\geq 0}$, and let $\ell = \left \lceil \log_2(m+j) \right \rceil$. We have
\begin{equation*}
\mathfrak{K}_j(m) < 2^{\ell + 1} + \frac{8 \cdot 2^{\ell}-j}{m}.
\end{equation*}
\end{lemma}

\begin{proof}
Suppose $m = 32n-3$. Let $N$ be the number of 1's in the binary expansion of $n$. It is straightforward to verify that the binary expansion of $m+6 = 32n+3$ has $N+2$ 1's. Similarly, we obtain the following table:
\begin{center}
\begin{tabular}{c | c}
$i$ & Number of 1's in Binary Expansion of $i$ \\ \hline 
$m+6$ & $N+2$ \\
$m+7$ & $N+1$ \\
$2m+6$ & $N$ \\
$2m+7$ & $N+1$
\end{tabular}
\end{center}
Consequently, we have that $\mathbf{t}_{m+7} \mathbf{t}_{m+8} = \mathbf{t}_{2m+7} \mathbf{t}_{2m+8}$. It follows that 
\begin{align}
\langle (m+6) 2^{\ell} + 1, (m+8) 2^{\ell} \rangle &= \mu^{\ell}(\mathbf{t}_{m+7} \mathbf{t}_{m+8}) \\
& = \mu^{\ell}(\mathbf{t}_{2m+7} \mathbf{t}_{2m+8} ) = \langle (2m+6) 2^{\ell} + 1, (2m+8) 2^{\ell} \rangle .
\end{align}
Applying Lemma \ref{lem-aid} with $s = 7$ and $a = 1$ gives that there exists $r \in \mathbb{Z}^{\geq 0}$ such that
\begin{equation} \label{eq:peter}
    2^{\ell} \cdot 6 + 1 \leq rm+j+1 < (r+1)m+j < 2^{\ell}\cdot 8.
\end{equation}
Therefore, we can write
\begin{equation} \label{eq:2} 
\langle 2^{\ell} \cdot 6 +1, 2^{\ell } \cdot 8 \rangle = w \langle rm+j+1, (r+1)m +j \rangle z,
\end{equation}
where $w = \langle 2^{\ell} \cdot 6+1, rm+j \rangle$ and $z = \langle (r+1)m+j+1, 2^{\ell} \cdot 8 \rangle$. Adding $2^{\ell} m$ everywhere in (\ref{eq:peter}) similarly gives that  
\begin{equation} \label{eq:5}
\langle 2^{\ell}(m+6) + 1, 2^{\ell}(m+8) \rangle = w' \langle (r+2^{\ell})m+j+1, (r+2^{\ell}+1)m+j \rangle z',
\end{equation}
where $w' = \langle 2^{\ell}(m+6)+1, (r+2^{\ell})m+j \rangle$ and $z' = \langle (r+2^{\ell}+1)m+j+1, 2^{\ell}(m+8) \rangle$. In the same way, adding $2^{\ell +1}m$ everywhere in (\ref{eq:peter}) gives that
\begin{equation} \label{eq:6}
\langle 2^{\ell}(2m+6) +1, 2^{\ell}(2m+8) \rangle = w'' \langle (r+2^{\ell+1})m + j +1, (r+2^{\ell+1}+1)m +j \rangle z'',
\end{equation}
where $w'' = \langle 2^{\ell}(2m+6)+1, (r+2^{\ell+1})m+j \rangle$ and $z'' = \langle (r+2^{\ell+1}+1)m+j+1, 2^{\ell}(2m+8) \rangle$. Observe that $\vert w'' \vert = rm + j -6 \cdot 2^{\ell} = \vert w' \vert$. Therefore, (\ref{eq:5}) and (\ref{eq:6}) imply
\begin{equation}
\langle (r+2^{\ell})m +j+1, (r+2^{\ell}+1)m +j \rangle = \langle (r+ 2^{\ell+1})m +j+1, (r+2^{\ell+1}+1)m+j \rangle.
\end{equation}
Noting from (\ref{eq:peter}) that $\displaystyle r+1 < \frac{8 \cdot 2^{\ell} - j}{m}$, we get
\begin{equation}
\mathfrak{K}_j(m) \leq r+2^{\ell +1} +1 < 2^{\ell +1} + \frac{8 \cdot 2^{\ell} - j}{m},
\end{equation}
as desired.
\end{proof}

\begin{remark}
We make note of an error in Defant's proof of an upper bound for $\mathfrak{K}_0(m)$ in the case $m \equiv 29 \pmod{32}$. In Defant's proof of \cite[Lemma 14]{Defant}, he claims that 
\begin{equation} \label{error:defant}
\bigcup_{r = 9}^{17} \left ( \frac{17}{2r}, \frac{10}{r+1} \right ) = \left ( \frac{1}{2} , 1 \right ),
\end{equation}
which implies the existence of some $r \in \{ 9,10, \ldots, 17 \}$ such that $\displaystyle \frac{17}{2r} < \frac{m}{2^{\ell}} < \frac{10}{r+1}$, where $\ell = \left \lceil \log_2 m \right \rceil$. However, (\ref{error:defant}) is in fact false. This mistake can be highlighted by observing that for $m = 32 \cdot 15 - 3 = 477$, there does not exist $r \in \{ 9,10, \ldots, 17 \}$ satisfying the desired inequality. Fortunately, setting $j = 0$ in Lemma \ref{29mod32} gives the bound $\displaystyle \mathfrak{K}_0(m) < 2^{\ell +1} + \frac{8 \cdot 2^{\ell}}{m}$, which is only slightly worse than Defant's intended bound of $\mathfrak{K}_0(m) \leq 2^{\ell} + 18$. In the same way as the error noted in Remark \ref{rmk:error2defant}, this error worsens Defant's lower bounds for $\displaystyle \liminf_{k \rightarrow \infty} (\gamma_0(k)/k)$ and $\displaystyle \limsup_{k \rightarrow \infty} (\gamma_0(k)/k)$, but we still know Defant's claimed lower bounds to be true \cite{Narayanan}. 
\end{remark}

Finally, we consider the case in which $m$ is an odd positive integer with $m \not \equiv 1 \pmod{8}$ and $m \not \equiv 29 \pmod{32}$.

\begin{lemma}[{cf.~\cite[Lemma 14]{Defant}}] \label{gen4.3}
Let $m$ be an odd positive integer with $m \not \equiv 1 \pmod 8$ and $m \not \equiv 29 \pmod{32}$. Let $j \in \mathbb{Z}^{\geq 0}$, and let $\ell = \left \lceil \log_2(m+j) \right \rceil$. We have
\begin{align*}
\mathfrak{K}_j(m) < 2^{\ell} + \frac{37 \cdot 2^{\ell} - j}{m}.
\end{align*}
\end{lemma}

\begin{proof}
Defant's proof of \cite[Lemma 14]{Defant} (up to where he considers the case $m \equiv 29 \pmod{32}$) applies almost exactly: \cite[Lemma 12]{Defant} and \cite[Lemma 13]{Defant} should merely be replaced by Lemma \ref{gen4.1} and Lemma \ref{gen4.2}, respectively. 
\end{proof}

The following two lemmas use the preceding results to establish a single upper bound for $\mathfrak{K}_j(m)$ for any integer $m \geq 2$.

\begin{lemma}[{cf.~\cite[Lemma 17]{Defant}}] \label{1mod8-asym}
Let $j \in \mathbb{Z}^{\geq 0}$, and suppose $m = 2^L h +1$, where $L$ and $h$ are integers with $L \geq 3$ and $h$ odd. Let $\ell = \left \lceil \log_2(m+j) \right \rceil$. Then 
\begin{align*}
\mathfrak{K}_j(m) \leq 2^{\ell} + \frac{2^{\ell +1}(2^{\ell}+2+j)}{m}.
\end{align*}
\end{lemma}

\begin{proof}
First, assume that $\displaystyle m+j > \left ( 1- \frac{1}{2^L-4} \right ) 2^{\ell}$. Observe that $2^{\ell} - 2^L h = 2^{\ell} -m+1$. Since $L < \ell$, we have that $2^L$ divides $2^{\ell} -2^L h$, which further gives that $2^L$ divides $2^{\ell} -m +1$. Since $2^{\ell} -m+1 > 0$, this gives that
\begin{align}
2^L \leq 2^{\ell} -m+1 < 2^{\ell} - \left ( 2^{\ell} - \frac{2^{\ell}}{2^L-4} -j \right ) +1 = \frac{2^{\ell}}{2^L-4} +j +1.
\end{align}
This implies that $2^{2L} - 4 \cdot 2^L < 2^{\ell} + j(2^L-4)  + 2^L -4$. Rearranging and dividing by $2^L$ gives the first inequality of
\begin{equation} \label{calc}
2^L < 2^{\ell - L} + (j+5) - 4(j+1) 2^{-L} < 2^{\ell -L +2} + 2^{\ell} -m - 4(j+1)2^{-L};
\end{equation}
the second inequality is straightforward to verify. From Lemma \ref{gen4.4}, we have that $\displaystyle \mathfrak{K}_j(m) < 2^{\ell} + \frac{2^{\ell}(2^{L+1}+4)-j}{m}$. Incorporating (\ref{calc}), we get
\begin{align}
2^{\ell}(2^{L+1}+4) -j & = 2^{\ell +1} \cdot 2^{L} + 2 \cdot 2^{\ell +1} -j \\
& < 2^{\ell +1} (2^{\ell -L+2} +2^{\ell} -m -4(j+1)2^{-L}) + 8 \cdot 2^{\ell -1} -j \\
& \leq (2^{\ell}-1)2^{\ell -L+3} + (2^{\ell +1}+8)2^{\ell -1} + (2^{\ell +1} - 2^{\ell -L +3}-1)j \\
& \leq (2^{\ell+1}+3)2^{\ell} + (2^{\ell +1}-15)j,
\end{align}
where, in the last step, we have used that $\ell = \left \lceil \log_2(m+j) \right \rceil \geq L+1$ and that $L \geq 3$. It follows that
\begin{equation}
\mathfrak{K}_j(m) < 2^{\ell} + \frac{(2^{\ell+1}+3)2^{\ell} + (2^{\ell +1}-15)j}{m} \leq 2^{\ell} + \frac{2^{\ell +1}(2^{\ell}+2+j)}{m}.
\end{equation}

Next, assume that $\displaystyle m+j \leq \left ( 1- \frac{1}{2^L-4} \right )2^{\ell}$ and $L \geq 4$. Let $n$ be the largest integer such that $m-n \equiv 2 \pmod 4$ and $n \leq 2^{L-1}$. Since $n \geq 2^{L-1}-3$, we have that $\displaystyle m+j \leq \left ( 1- \frac{1}{2n+2} \right )2^{\ell}$. By the condition $m-n \equiv 2 \pmod 4$, we have $\mathbf{t}_{m-n} = \mathbf{t}_{m-n+1}$. We can, therefore, apply Lemma \ref{gen4.5}, which gives
\begin{equation}
\mathfrak{K}_j(m) \leq 2^{\ell +1} - \frac{2^{\ell +1}(n-\frac{1}{2}) +j}{m} < 2^{\ell +1} - \frac{2^{\ell +1}(2^{L-1}-4)}{m} \leq 2^{\ell} + \frac{2^{\ell +1}(2^{\ell}+2+j)}{m}.
\end{equation}
Finally, suppose $L=3$. By Lemma \ref{gen4.4},
\begin{equation}
\mathfrak{K}_j(m) < 2^{\ell} + \frac{20 \cdot 2^{\ell}-j}{m} < 2^{\ell} + \frac{2^{\ell +1}(2^{\ell}+2+j)}{m}. 
\end{equation}
\end{proof}

\begin{lemma} \label{other-asym}
Let $j,m \in \mathbb{Z}^{\geq 0}$ with $m \geq 2$ and $m \not \equiv 1 \pmod 8$. Let $\ell = \left \lceil \log_2(m+j) \right \rceil$. Then
\begin{align*}
\mathfrak{K}_j(m) \leq 2^{\ell} + \frac{2^{\ell +1} \cdot \max \{  2^{\ell}+2+j, 20 \}}{m}.
\end{align*}
\end{lemma}

\begin{proof}
If $m \equiv 0 \pmod 2$, we have by Lemma \ref{even} that
\begin{equation}
\mathfrak{K}_j(m) < 2^{\ell +1} + \frac{2^{\ell +1} -j}{m} < 2^{\ell} + \frac{2^{\ell +1}(2^{\ell}+2+j)}{m}.
\end{equation}
If $m \equiv 29 \pmod{32}$, we have by Lemma \ref{29mod32} that
\begin{equation}
\mathfrak{K}_j(m) < 2^{\ell + 1} + \frac{8 \cdot 2^{\ell}-j}{m} < 2^{\ell} + \frac{2^{\ell +1}(2^{\ell}+2+j)}{m}.
\end{equation}
Finally, if $m$ is an odd positive integer with $m \not \equiv 1 \pmod 8$ and $m \not \equiv 29 \pmod{32}$, we have by Lemma \ref{gen4.3} that
\begin{equation}
\mathfrak{K}_j(m) < 2^{\ell} + \frac{37 \cdot 2^{\ell} - j}{m} < 2^{\ell} + \frac{20 \cdot 2^{\ell +1}}{m}. 
\end{equation}
\end{proof}

We are now ready to prove the lower bounds for $\displaystyle \liminf_{k \rightarrow \infty} ( \gamma_j(k)/k )$ and $\displaystyle \limsup_{k \rightarrow \infty} ( \gamma_j(k)/k )$.

\begin{thm}[{cf.~\cite[Theorem 18]{Defant}}]
For any nonnegative integer $j$,
\begin{align*}
\liminf_{k \rightarrow \infty} \frac{\gamma_j(k)}{k} \geq \frac{1}{10} \hspace{.5cm} \text{and} \hspace{.5cm} \limsup_{k \rightarrow \infty} \frac{\gamma_j(k)}{k} \geq \frac{1}{5}.
\end{align*}
\end{thm}

\begin{proof}
Fix $j \in \mathbb{Z}^{\geq 0}$. For sufficiently large $\ell \in \mathbb{Z}^+$ (precisely, for $\ell$ large enough so that $2^{\ell -1}-j > 0$), define $\displaystyle g_j(\ell) = 2^{\ell} + \frac{2^{\ell +1} \cdot \max \{ 2^{\ell}+2+j,20 \}}{2^{\ell -1}-j}$. Choose some $k \in \mathbb{Z}^+$ large enough so that $\log_2(\gamma_j(k)) > j$. Let $\ell = \left \lceil \log_2(\gamma_j(k) + j) \right \rceil$. By definition of $\gamma_j$, we have that $k < \mathfrak{K}_j(\gamma_j(k))$. Applying Lemmas \ref{1mod8-asym} and \ref{other-asym} gives $\displaystyle \frac{\gamma_j(k)}{k} > \frac{\gamma_j(k)}{g_j(\ell)} > \frac{2^{\ell -1}-j}{g_j(\ell)}$. Therefore, $\displaystyle \liminf_{k \rightarrow \infty} \frac{\gamma_j(k)}{k} \geq \lim_{\ell \rightarrow \infty} \frac{2^{\ell -1}-j}{g_j(\ell)} = \frac{1}{10}$.

By Lemmas \ref{1mod8-asym} and \ref{other-asym}, we have that $\mathfrak{K}_j(m) < \left \lfloor g_j(\ell) \right \rfloor +1$ for all positive integers $m < 2^{\ell} -j$. Therefore, by the definition of $\gamma_j$, we have that $\gamma_j(\left \lfloor g_j(\ell ) \right \rfloor +1 ) \geq 2^{\ell} -j +1$. Consequently,
\begin{equation}
\limsup_{k \rightarrow \infty} \frac{\gamma_j(\ell)}{k} \geq \limsup_{\ell \rightarrow \infty} \frac{\gamma_j(\left \lfloor g_j(\ell ) \right \rfloor +1  )}{\left \lfloor g_j(\ell ) \right \rfloor +1 } \geq \lim_{\ell \rightarrow \infty} \frac{2^{\ell} -j +1}{g_j(\ell)+1} = \frac{1}{5}. 
\end{equation}
\end{proof}

\subsection{Upper Bounds for $\gamma_j(k)/k$}

In this subsection we establish upper bounds for $\displaystyle \liminf_{k \rightarrow \infty} (\gamma_j(k)/k)$ and $\displaystyle \limsup_{k \rightarrow \infty} (\gamma_j(k)/k)$. We start by stating a result of Defant.

\begin{prop}[{\cite[Proposition 6]{Defant}}] \label{prop3.1}
Let $m \geq 2$ be an integer, and let $\delta (m) = \left \lceil \log_2(m/3) \right \rceil$. If $y$ and $v$ are words such that $yvy$ is a factor of $\mathbf{t}$ and $\vert y \vert = m$, then $2^{\delta(m)}$ divides $\vert yv \vert$.
\end{prop}

We proceed with a lemma and theorem whose proofs closely follow those of \cite[Lemma 19]{Defant} and \cite[Theorem 20]{Defant}, respectively.

\begin{lemma}[{cf.~\cite[Lemma 19]{Defant}}] \label{gen4.7}
For each integer $\ell \geq 3$ and any nonnegative integer $j$, we have 
\begin{align*}
\mathfrak{K}_j(3 \cdot 2^{\ell -2}+1) > \frac{5 \cdot 2^{2 \ell -3}-j}{3 \cdot 2^{\ell -2}+1} \hspace{.5cm} \text{and} \hspace{.5cm} \mathfrak{K}_j(2^{\ell -1} +3) > \frac{2^{2 \ell -2}-j}{2^{\ell -1}+3}.
\end{align*}
\end{lemma}

\begin{proof}
Fix $\ell \geq 3$ and $j \in \mathbb{Z}^{\geq 0}$. Let $m = 3 \cdot 2^{\ell -2} +1$ and $m' = 2^{\ell -1} +3$. By the definitions of $\mathfrak{K}_j(m)$ and $\mathfrak{K}_j(m')$, there exist nonnegative integers $r < \mathfrak{K}_j(m) -1$ and $r' < \mathfrak{K}_j(m')-1$ such that
\begin{equation}
\langle rm+j+1, (r+1)m+j \rangle = \langle (\mathfrak{K}_j(m)-1)m+j+1, \mathfrak{K}_j(m)m+j \rangle
\end{equation}
and 
\begin{equation}
\langle r'm' +j+1, (r'+1)m'+j \rangle = \langle (\mathfrak{K}_j(m')-1)m' +j+1, \mathfrak{K}_j(m')m' + j \rangle. 
\end{equation}
Following the corresponding part of the proof of \cite[Lemma 19]{Defant} and using Proposition \ref{prop3.1}, we may assume that $\mathfrak{K}_j(m) = r + 2^{\ell -1}+1$ and that $\mathfrak{K}_j(m') = r' + 2^{\ell -2} +1$.

Assume for the sake of contradiction that $\displaystyle \mathfrak{K}_j(m) \leq \frac{5 \cdot 2^{2 \ell -3}-j}{m}$. Let $u = \langle rm+j+1, (r+1)m+j \rangle$ and $v = \langle (\mathfrak{K}_j(m)-1)m+j+1, \mathfrak{K}_j(m)m+j \rangle$. Following the corresponding part of the proof of \cite[Lemma 19]{Defant} and using the fact that $\t$ is overlap-free, we get that $u \neq v$, a contradiction.

Assume next that $\displaystyle \mathfrak{K}_j(m') \leq \frac{2^{2 \ell -2}-j}{m'}$. Let $u' = \langle r'm'+j+1, (r'+1)m'+j \rangle$ and $v' = \langle (\mathfrak{K}_j(m')-1)m' +j+1, \mathfrak{K}_j(m')m'+j \rangle $. Let $\displaystyle q = \left \lceil \frac{r'm'+j+1}{2^{\ell -2}} \right \rceil$ and $\displaystyle H = \min \left \{ (r'+1)m', (q+2)2^{\ell -2} +j \right \}$. Set $U = \langle r'm'+j+1, H+j \rangle $ and $V= \langle (r'+2^{\ell -2})m'+j+1, H+2^{\ell -2}m'+j \rangle $. Note that the word $U$ is the prefix of $u'$ of length $H-r'm'$. Recalling that $\mathfrak{K}_j(m') = r' + 2^{\ell -2} +1$, we see that $V$ is the prefix of $v'$ of length $H-r'm'$. Since $u'=v'$, it follows that $U=V$. Now, there are words $w'$, $z'$, $w''$, and $z''$ such that we have the following:
\begin{align}
    \mu^{\ell -2} (\mathbf{t}_{q} \mathbf{t}_{q+1} \mathbf{t}_{q+2}) &= \langle (q-1)2^{\ell -2}+1, (q+2)2^{\ell -2} \rangle = w' U z',\\
    \mu^{\ell -2} (\mathbf{t}_{q+m'} \mathbf{t}_{q+m'+1} \mathbf{t}_{q+m'+2}) &= \langle (q+m'-1)2^{\ell -2}+1, (q+m'+2)2^{\ell -2} \rangle = w''V z''.
\end{align}
Using that $U=V$ and following the corresponding part of the proof of \cite[Lemma 19]{Defant}, we get that
\begin{equation}
0 \leq \vert w' \vert = \vert w'' \vert = r'm'+j-(q-1)2^{\ell -2} \leq r'm' +j - \left ( \frac{r'm'+j+1}{2^{\ell -2}} -1 \right )2^{\ell -2} < 2^{\ell -2},
\end{equation}
and hence that $\t_q = \t_{q+m'}$. Note also that $\vert z' \vert = \vert z'' \vert = (q+2)2^{\ell -2} - (H+j)$. We show that $H+2^{\ell -2}m+j+1 - (q+m'+1)2^{\ell -2} > 0$, which will show that $z''$ is a suffix of $\mu^{\ell -2}(\mathbf{t}_{q+m'+2})$. Observe that
\begin{align}
H+2^{\ell -2}m' +j+1 - (q+m'+1)2^{\ell -2} &= H+j+1 - q2^{\ell -2} - 2^{\ell -2} \\
& > H+j+1 - \left ( \frac{r'm'+j+1}{2^{\ell -2}} +1 \right ) 2^{\ell -2} - 2^{\ell -2} \\
&= H-r'm' - 2^{\ell -1}.
\end{align}
If $H= r'm' + m'$, then $H = r'm' +2^{\ell -1}+3 > r'm'+2^{\ell -1}$, giving $H-r'm' - 2^{\ell -1} > 0$. Alternatively, if $H = (q+2)2^{\ell -2} -j$, then we have
\begin{equation}
(q+2)2^{\ell -2} - j \geq \left ( \frac{r'm'+j+1}{2^{\ell -2}} +2 \right ) 2^{\ell -2} - j = r'm' +2^{\ell -1} +1 > r'm' + 2^{\ell -1},
\end{equation}
and again $H-r'm' - 2^{\ell -1} > 0$. It follows that $\mathbf{t}_{q+2} = \mathbf{t}_{q+m'+2}$. Similarly, $\mathbf{t}_{q+1} = \mathbf{t}_{q+m'+1}$. 

Now,
\begin{equation}
r' = \mathfrak{K}_j(m') -2^{\ell -2} -1 \leq \frac{2^{2 \ell -2} -j}{m'} - 2^{\ell -2} -1 = \frac{2^{2 \ell -3} - 5 \cdot 2^{\ell -2} -j -3}{m'}.
\end{equation}
It follows that $r'm' +j+1 \leq 2^{2 \ell -3} - 5 \cdot 2^{\ell -2} -2$, which gives that $\displaystyle \frac{r'm'+j+1}{2^{\ell -2}} \leq 2^{\ell -1} -5$. Therefore, $q+4 < 2^{\ell -1}$. Consequently, for each $s \in \{ 0,1,2 \}$, the binary expansion of $q+m'+s-1$ has exactly one more 1 than the binary expansion of $q+s+2$. Thus, 
\begin{equation}
\mathbf{t}_{q+3} \mathbf{t}_{q+4} \mathbf{t}_{q+5} = \overline{\mathbf{t}_{q+m'} \mathbf{t}_{q+m'+1} \mathbf{t}_{q+m'+2}} = \overline{\mathbf{t}_q \mathbf{t}_{q+1} \mathbf{t}_{q+2}}.
\end{equation}
As in the proof of \cite[Lemma 19]{Defant}, this contradicts the fact that $\t$ is cube-free.
\end{proof}

\begin{thm}[{cf.~\cite[Theorem 20]{Defant}}]
For any nonnegative integer $j$, 
\begin{equation*}
\liminf_{k \rightarrow \infty} \frac{\gamma_j(k)}{k} \leq \frac{9}{10} \hspace{.5cm} \text{ and } \hspace{.5cm} \limsup_{k \rightarrow \infty} \frac{\gamma_j(k)}{k} \leq \frac{3}{2}.
\end{equation*}
\end{thm}

\begin{proof}
This proof works in exactly the same way as the proof of \cite[Theorem 20]{Defant} with $\displaystyle f_j(\ell) = \left \lfloor \frac{5 \cdot 2^{2 \ell -3} -j}{3 \cdot 2^{\ell -2} +1} \right \rfloor$, $\displaystyle h_j(\ell) = \left \lfloor \frac{2^{2 \ell -2}-j}{2^{\ell -1}+3} \right \rfloor$, and Lemma \ref{gen4.7} in place of $f(\ell)$, $h(\ell)$, and \cite[Lemma 19]{Defant}, respectively. 
\end{proof}

\section{Asymptotics for $\Gamma_j(k)$} \label{sec:Gamma}

Having established asymptotic bounds showing that $\gamma_j(k)$ grows linearly in $k$, we now turn our attention to $\Gamma_j(k)$. In this section, we prove that $\displaystyle \liminf_{k \rightarrow \infty} (\Gamma_j(k)/k) = 3/2$ and $\displaystyle \limsup_{k \rightarrow \infty} (\Gamma_j(k)/k) = 3$. We start by motivating our definition of $\Gamma_j(k)$.

Recall that we have defined $\Gamma_j(k) := \sup ((2 \mathbb{Z}^+ -1) \setminus \mathcal{F}_j(k))$. Also recall that Defant's motivation for defining $\Gamma_0(k) := \sup ((2 \mathbb{Z}^+ -1) \setminus \mathcal{F}_0(k))$ is the property that $m \in AP_0(\t,k)$ if and only if $2m \in AP_0(\t,k)$, meaning that the only interesting elements of $AP_0(\t,k)$ are those that are odd. However, as previously noted, it is not necessarily the case for nonzero $j$ that $m \in AP_j(\t,k)$ if and only if $2m \in AP_j(\t,k)$. As such, it is not initially clear that we are motivated in generalizing Defant's definition of $\Gamma_0(k)$ in the way we have. In other words, if even elements of $AP_j(\t,k)$ can be interesting, why would we consider only the odd elements? The following proposition demonstrates a drawback of considering all even elements of $AP_j(\t ,k )$.

\begin{prop} \label{prop:zero}
For $k \geq 3$, the set $2 \mathbb{Z}^+ \setminus (AP_0(\t,k) \cap 2 \mathbb{Z}^+)$ is unbounded.
\end{prop}

\begin{proof}
Since $\t_1 \t_2 \cdots \t_9 = 011010011$ has two occurrences of 011, we have that $3 \in \mathbb{Z}^{+} \setminus AP_0(\t,k)$ for all $k \geq 3$. Recall that $m \in AP_0(\t,k)$ if and only if $2m \in AP_0(\t,k)$. Therefore, $3 \cdot 2^L \in 2 \mathbb{Z}^+ \setminus (AP_0(\t,k) \cap 2 \mathbb{Z}^+)$ for all $L \in \mathbb{Z}^+$. The proposition follows.
\end{proof}

As a consequence of Proposition \ref{prop:zero}, if we were to include even numbers by defining $\Gamma_j(k) := \sup ( \mathbb{Z}^+ \setminus AP_j(\t,k))$, we would have that $\Gamma_0(k) = \infty$ for $k \geq 3$, which is contrary to the result we are trying to generalize (namely, that $\Gamma_0(k)$ grows linearly in $k$). Corollary \ref{conj} below shows that only finitely many even elements of $AP_j(\t,k)$ are interesting, and consequently further motivates our definition of $\Gamma_j(k)$. To prove Corollary \ref{conj}, we first require a lemma and proposition, both of which were suggested by an anonymous referee. The lemma follows closely from a much more general result in \cite{Queff} regarding ``recognizability" in certain sequences. 

\begin{lemma} \label{lem:referee}
Fix $x \in \mathbb{Z}^{+}$. Then there exists $N \in \mathbb{Z}^{+}$ such that $\langle n_1+1 , n_1 + N \rangle = \langle n_2 + 1, n_2 + N \rangle$ implies $2^x \mid (n_1 - n_2)$.  
\end{lemma}

\begin{proof}
\cite[Lemma 5.6]{Queff} gives the following: For all $x \in \mathbb{Z}^{+}$, there exists $N \in \mathbb{Z}^{+}$ such that if $n$ is a nonnegative multiple of $2^x$ and $\langle n+1, n + N \rangle = \langle n' + 1, n' + N \rangle$, then $n'$ is a nonnegative multiple of $2^x$. In the language of \cite{Queff}, each $\mu^x$ is ``recognizable." 

Now, suppose we are given $n_1, n_2 \in \mathbb{Z}^{\geq 0}$. Choose the smallest nonnegative multiple of $2^x$ that is greater than $n_1$ (say, $n_1+r$). By \cite[Lemma 5.6]{Queff}, there exists $N \in \mathbb{Z}^+$ such that if $\langle n_1+r+1 , n_1 + N \rangle = \langle n_2 +r+ 1, n_2 + N \rangle$, then $n_2 + r$ is a nonnegative multiple of $2^x$. With this choice of $N$, it follows that if $\langle n_1+1 , n_1 + N \rangle = \langle n_2 + 1, n_2 + N \rangle$, then $2^x$ divides $(n_1+r) - (n_2+r) = n_1 - n_2$.
\end{proof}

\begin{prop} \label{prop:referee}
Let $J,k \in \mathbb{Z}^{+}$. Then there exists $N \in \mathbb{Z}^{+}$ such that for all $m \geq N$ and all $1 \leq j,j' \leq J$, we have that $m \in AP_j(\t,k)$ if and only if $m \in AP_{j'}(\t,k)$.
\end{prop}

\begin{proof}
Take $a,b \in \mathbb{Z}^+$ such that $J < 2^a$ and $k < 2^b$. Set $x := a+b$. Choose $N \in \mathbb{Z}^+$ corresponding to $x$, as provided by Lemma \ref{lem:referee}. 

Assume that $m \not \in AP_j(\t,k)$ for some $m \geq N$. Then, by definition, $\langle j+1 , j+km \rangle$ is not a $k$-anti-power, meaning there exist $\ell_1 \neq \ell_2 \in  \{ 0, \ldots , k-1 \}$ such that 
\begin{equation}
\langle j+ \ell_1 m +1 , j+(\ell_1 + 1)m \rangle = \langle j+ \ell_2 m + 1, j + (\ell_2 + 1)m \rangle.
\end{equation}
By our choices of $N$ and $m$, Lemma \ref{lem:referee} gives that $2^x \mid (\ell_2 - \ell_1)m$. Since $\vert \ell_2 - \ell_1 \vert < k < 2^b$, this shows that $2^a \mid m$. 

For $i=1,2$, $\langle j+ \ell_i m + 1, j + (\ell_i +1)m \rangle $ can be broken into the partial block $\langle j+\ell_im + 1, \ell_i m + 2^a \rangle$, then some blocks of length $2^a$, and finally a partial block $\langle (\ell_i + 1)m + 1, j+ (\ell_i + 1)m \rangle $. By assumption, all of these blocks coincide for $i=1,2$. Recall that we can view $\t$ as a word over the alphabet $\{ A_a, B_a \}$, where $\vert A_a \vert = \vert B_a \vert = 2^a$. Since $A_a$ and $B_a$ differ at every position, we see that $\langle \ell_i m +1, \ell_i m + 2^a \rangle$ and $\langle (\ell_i + 1)m+1, (\ell_i + 1)m + 2^a \rangle$ must coincide for $i = 1,2$ as well. Putting this all together, we have that
\begin{equation}
\langle \ell_1 m + 1, (\ell_1 + 1)m + 2^a \rangle = \langle \ell_2 m +1, (\ell_2+1)m + 2^a \rangle .
\end{equation}
Thus, $m \not \in AP_{j'}(\t,k)$ for any $1 \leq j' < 2^a$. Reversing the roles of $j$ and $j'$ completes the proof.
\end{proof}

\begin{cor} \label{conj}
For any fixed $j,k \in \mathbb{Z}^{\geq 0}$ with $k \geq 3$, the statement
\begin{equation*}
m \in AP_j(\t,k) \iff 2m \in AP_j(\t,k)
\end{equation*}
holds for all but finitely many $m \in \mathbb{Z}^+$.
\end{cor}

\begin{proof}
Recall that $\t$ is fixed under the morphism $\mu : \{ 0,1 \}^{\leq \omega} \rightarrow \{ 01,10 \}^{\leq \omega}$ uniquely defined by $\mu (0) = 01$ and $\mu (1) = 10$. With this in mind, it is easily seen that $m \in AP_j (\t,k)$ if and only if $2m \in AP_{2j}(\t,k)$. Moreover, applying Proposition \ref{prop:referee}, we have that for sufficiently large $m$, $2m \in AP_{2j}(\t,k)$ if and only if $2m \in AP_j(\t,k)$. Together, this shows that for all but finitely many $m \in \mathbb{Z}^+$, $m \in AP_j(\t,k)$ if and only if $2m \in AP_j(\t,k)$  
\end{proof}

Having motivated our definition of $\Gamma_j(k)$, let us proceed by proving a Corollary to \cite[Proposition 6]{Defant} (stated above as Proposition \ref{prop3.1}).

\begin{cor}[{cf.~\cite[Corollary 7]{Defant}}] \label{gencor}
Let $m,k \in \mathbb{Z}^+$, where $m \in (2 \mathbb{Z}^+ -1) \setminus \mathcal{F}_j(\mathbf{t},k)$ and $k \geq 3$. Let $\delta(m) = \left \lceil \log_2(m/3) \right \rceil$. Then $k-1 \geq 2^{\delta (m)}$. 
\end{cor}

\begin{proof}
By the hypotheses of the corollary, we have that the $j$-fix of $\t$ of length $km$ is not a $k$-anti-power. It follows that there exist integers $n_1$ and $n_2$ with $0 \leq n_1 < n_2 \leq k-1$ such that
\begin{equation}
\langle n_1 m +j+1, (n_1+1)m+j \rangle = \langle n_2 m + j+1, (n_2+1)m+j \rangle.
\end{equation}
Using Proposition \ref{prop3.1}, the remainder of the proof follows easily from the proof of \cite[Corollary 7]{Defant}.
\end{proof}

We now present a technical lemma that will be useful for constructing identical pairs of subwords of the Thue-Morse word. These pairs of subwords will allow us to establish upper bounds on $\mathfrak{K}_j(m)$ for certain odd values of $m$. It will be useful to keep in mind that $\Gamma_j(k) \geq m$ whenever $k \geq \mathfrak{K}_j(m)$; this fact follows from Definitions \ref{def:Gamma} and \ref{def:K}. 

\begin{lemma}[{cf.~\cite[Lemma 8]{Defant}}] \label{gen3.1}
Suppose that $\ell \geq 2,\, 2 \leq m < 2^{\ell}, \, r, \,  h, \, p, \, q$ are nonnegative integers satisfying the following conditions:
\begin{itemize}
\item $h < 2^{\ell -2}$
\item $rm = 2^{\ell +1} p + 2^{\ell -1} + h - j$
\item $(r+1)m \leq 2^{\ell +1} p + 5 \cdot 2^{\ell -2} -j$
\item $(r+2^{\ell -2})m = 2^{\ell +1} q + 3 \cdot 2^{\ell -2} +h -j$
\item $\mathbf{t}_{p+1} \neq \mathbf{t}_{q+1}$
\end{itemize}
Then $\langle rm+j+1, (r+1)m+j \rangle = \langle (r+2^{\ell -2})m+1, (r+2^{\ell -2}+1)m \rangle$, and $\mathfrak{K}_j(m) \leq r+2^{\ell -2} +1$.
\end{lemma}

\begin{proof}
Let $u = \langle rm+j+1, (r+1)m+j \rangle$ and $v = \langle (r+2^{\ell -2})m+j+1, (r+2^{\ell -2}+1)m +j \rangle$. Following the proof of \cite[Lemma 8]{Defant} almost exactly (replacing his variables and conditions with the corresponding ones established above), we get that $u=v$. It follows that the $j$-fix of $\mathbf{t}$ of length $(r+2^{\ell -2}+1)m$ is not a $(r+2^{\ell -2}+1)$-anti-power, meaning $\mathfrak{K}_j(m) \leq r+2^{\ell -2}+1$.
\end{proof}

We are now ready to prove one of the two main results of this section, the proof of which adapts a construction from the proof of \cite[Theorem 9]{Defant}.

\begin{thm}[{cf.~\cite[Theorem 9]{Defant}}] \label{thm:gen3.1}
Fix $j \in \mathbb{Z}^{\geq 0}$. For all integers $k \geq 3$, we have $\Gamma_j(k) \leq 3k-4$. Moreover, $\displaystyle \limsup_{k \rightarrow \infty} \frac{\Gamma_j(k)}{k} = 3$.  
\end{thm}

\begin{proof}
The proof that $\Gamma_j(k) \leq 3k-4$ follows almost exactly the corresponding part of the proof of \cite[Theorem 9]{Defant} (replacing the reference to \cite[Corollary 7]{Defant} with a reference to Corollary \ref{gencor}).

It remains to show that $\displaystyle \limsup_{k \rightarrow \infty} (\Gamma_j(k)/k) \geq 3$. For each positive integer $\alpha$, define $k_{\alpha} = 2^{2 \alpha} + 2^{\alpha }+2$. Fix an integer $\alpha \geq \left \lceil \log_2(j) \right \rceil +2$, and set $\ell = 2 \alpha +2$, $m=3 \cdot 2^{2 \alpha} - 2^{\alpha} + 1$, $r = 2^{\alpha}+1$, $h = j+1$, $p = 3 \cdot 2^{\alpha -3}$, and $q = 3 \cdot 2^{2 \alpha -3} + 2^{\alpha -2}$. Following the proof of \cite[Theorem 9]{Defant}, we see that we can apply Lemma \ref{gen3.1} to get that $\mathfrak{K}_j(m) \leq r + 2^{\ell -2} +1 = k_{\alpha}$. In other words, we have that the $j$-fix of $\mathbf{t}$ of length $k_{\alpha} m$ is not a $k_{\alpha}$-anti-power, meaning $\Gamma_j(k_{\alpha}) \geq m = 3 \cdot 2^{2 \alpha} - 2^{\alpha} +1$. It follows that 
\begin{equation}
\frac{\Gamma_j(k_{\alpha})}{k_{\alpha}} \geq  \frac{3 \cdot 2^{2 \alpha}-2^{\alpha}+1}{2^{2 \alpha} + 2^{\alpha}+2}
\end{equation}
for each $\alpha \geq \left \lceil \log_2(j) \right \rceil +2$. Consequently, $(k_{\alpha})_{\alpha \geq \left \lceil \log_2(j) \right \rceil +2}$ is an increasing sequence of positive integers with the property that $\Gamma_j(k_{\alpha})/k_{\alpha} \rightarrow 3$ as $\alpha \rightarrow \infty$. This shows that $\displaystyle \limsup_{k \rightarrow \infty} ( \Gamma_j(k)/k) \geq 3$, completing the proof.

\end{proof}

\begin{remark} \label{nonempty}
The construction in the previous theorem also functions to show that $(2 \mathbb{Z}^+ -1) \setminus \mathcal{F}_j(k)$ is nonempty for sufficiently large $k$. In particular, for $j>0$ and for any integer $\alpha \geq \left \lceil \log_2 (j) \right \rceil$, we have that $m = 3 \cdot 2^{2 \alpha} - 2^{\alpha} + 1 \in (2 \mathbb{Z}^+ -1) \setminus \mathcal{F}_j(k)$ for all $k \geq k_{\alpha} = 2^{2 \alpha} + 2^{\alpha} +2$. 
\end{remark}

Next, we present a lemma that will aid in the proof of the final main result of the paper. The lemma adapts constructions from \cite[Lemma 10]{Defant}, but it only applies for integers $j > 0$; \cite[Lemma 10]{Defant} gives the same result in the case that $j=0$.

\begin{lemma}[{cf.~\cite[Lemma 10]{Defant}}] \label{gen3.2}
Fix $j \in \mathbb{Z}^{+}$ and let $n$ be the number of 1's in the binary expansion of $j$. For integers $\alpha \geq \left \lceil \log_2(j) \right \rceil +2$, $\beta \geq \left \lceil \log_2(j) \right \rceil +9$, and $\rho \geq \left \lceil \log_2(j) \right \rceil +8$, define
\begin{equation*}
k_{\alpha} = 2^{2 \alpha}+2^{\alpha}+2 \hspace{.5cm} \text{and} \hspace{.5cm} K_{\beta} = 2^{2 \beta +1} + 3 \cdot 2^{\beta +3} + 49 \hspace{.5cm} \text{and} \hspace{.5cm} \kappa_{\rho} = 2^{\rho} +2.
\end{equation*}
We have $\Gamma_j(k_{\alpha}) \geq 3 \cdot 2^{2 \alpha} - 2^{\alpha} + 1$, $\Gamma_j(K_{\beta}) \geq 3 \cdot 2^{2 \beta +1} - 2^{\beta -1} +1$, and $\Gamma_j(\kappa_{\rho}) \geq 5 \cdot 2^{\rho -1} - 8 \chi_j (\rho) +1$, where 
\[
    \chi_j(\rho)=\left\{
                \begin{array}{ll}
                  2j+1, \hspace{.25cm} \text{ if } (n + \rho) \equiv 0 \pmod 2; \\
                  4j+3, \hspace{.25cm} \text{ if } (n + \rho) \equiv 1 \pmod 2. 
                \end{array}
              \right.
  \]
\end{lemma}

\begin{proof}
The lower bound for $\Gamma_j(k_{\alpha})$ was established in the proof of Theorem \ref{thm:gen3.1}. To bound $\Gamma_j(K_{\beta})$ from below, let $\ell = 2 \beta +3$, $m = 3 \cdot 2^{2 \beta +1} - 2^{\beta -1} +1$, $r = 3 \cdot 2^{\beta +3} + 48$, $h = 48 +j$, $p = 9 \cdot 2^{\beta} + 17$, and $q = 3 \cdot 2^{2 \beta -2} + 143 \cdot 2^{\beta -4} + 17$. Following the proof of \cite[Lemma 10]{Defant}, we see that we can apply Lemma \ref{gen3.1} to get that $\mathfrak{K}_j(m) \leq r+2^{\ell -2}+1 = K_{\beta}$, meaning the $j$-fix of $\mathbf{t}$ of length $K_{\beta} m$ is not a $K_{\beta}$-anti-power. Hence, $\Gamma_j(K_{\beta}) \geq m = 3 \cdot 2^{2 \beta +1} - 2^{\beta -1} +1$, as desired.

We now establish the lower bound for $\Gamma_j(\kappa_{\rho})$. Fix $\rho \geq \left \lceil \log_2(j) \right \rceil +8$. Define $\ell ' = \rho +2$, $m' = 5 \cdot 2^{\rho -1} - 8 \chi_j(\rho) +1$, $r' = 1$, $h' = 2^{\rho -1} - 8 \chi_j(\rho) +j+1$, $p' =0$, and $q' = 5 \cdot 2^{\rho -4} - \chi_j(\rho)$. It is straightforward to verify that these choices satisfy the first four of the five conditions of Lemma \ref{gen3.1}. To prove that $\t_{p'+1} \neq \t_{q'+1}$, we present an argument that depends on the parity of the number of 1's in the binary expansion of $j$ (which we have denoted by $n$). Assume that $n$ is odd; the case in which $n$ is even follows similarly. We consider two cases.

First, assume that $\rho \equiv 0 \pmod 2$. In this case, $\chi_j(\rho) = 4j +3$, so the binary expansion of $\chi_j(\rho)$ has $n+2$ 1's. Note that 
\begin{equation}
\left \lceil \log_2 \chi_j(\rho) \right \rceil = \left \lceil \log_2(4j+3) \right \rceil \leq 2 + \left \lceil \log_2(j+1) \right \rceil \leq 3 + \left \lceil \log_2 (j) \right \rceil < \rho -4.
\end{equation}	 
It follows that when right-justified, all of the 1's in the binary expansion of $5 \cdot 2^{\rho -4}$ are to the left of all the 1's in the binary expansion of $\chi_j(\rho)$. Binary subtraction thus shows that there are $\rho -4 -n$ 1's in the binary expansion of $5 \cdot 2^{\rho -4} - \chi_j(\rho)$. Since $n$ is odd and $\rho$ is even, we get that $\rho -4 -n $ is odd, meaning $\mathbf{t}_{q'+1} = 1 \neq 0 = \t_{p'+1}$. 

Next, assume instead that $\rho \equiv 1 \pmod 2$, meaning $\chi_j(\rho) = 2j+1$. In this case, the binary expansion of $\chi_j(\rho)$ has $n+1$ 1's. As before, binary subtraction shows that there are $\rho -3 -n$ 1's in the binary expansion of $5 \cdot 2^{\rho -4} - \chi_j(\rho)$. Since $n$ is odd and $\rho$ is even, we have that $\rho -3 -n$ is odd, meaning $\mathbf{t}_{q'+1} =1 \neq 0 = \t_{p'+1}$.
	
We have shown that $\ell'$, $m'$, $r'$, $h'$, $p'$, and $q'$ satisfy the conditions of Lemma \ref{gen3.1}. Applying the lemma gives that $\mathfrak{K}_{j}(m) \leq r' + 2^{\ell ' -2} + 1 = \kappa_{\rho}$. Therefore, $\Gamma_j(\kappa_{\rho}) \geq m = 5 \cdot 2^{\rho -1} - 8 \chi_j(\rho) +1$. This completes the proof.
\end{proof}

\begin{thm}[{cf.~\cite[Theorem 11]{Defant}}] \label{thm:gen3.2}
For any nonnegative integer $j$, $\displaystyle \liminf_{k \rightarrow \infty} \frac{\Gamma_j(k)}{k} = \frac{3}{2}$.
\end{thm}

\begin{proof}
The inequality $\displaystyle \liminf_{k \rightarrow \infty} (\Gamma_j(k)/k) \leq 3/2$ follows from the corresponding part of the proof of \cite[Theorem 11]{Defant} (replacing $\Gamma (k)$ with $\Gamma_j(k)$ and \cite[Corollary 7]{Defant} with Corollary \ref{gencor}).

It remains to show that $\displaystyle \liminf_{k \rightarrow \infty} (\Gamma_j(k)/k) \geq 3/2$. Recall the definitions of $k_{\alpha}$, $K_{\beta}$, $\kappa_{\rho}$, and $\chi_j(\rho)$ from Lemma \ref{gen3.2}. Let $\eta = 2 \left \lceil \log_2(j) \right \rceil +21$, fix $k \geq \kappa_{\eta}$, and put $m = \Gamma_j(k)$. Since $k \geq \kappa_{\eta}$, Lemma \ref{gen3.2} and the fact that $\Gamma_j$ is nondecreasing (see Remark \ref{nondecreasing}) together give $m = \Gamma_j(k) \geq \Gamma_j(\kappa_{\eta}) \geq 5 \cdot 2^{\eta -1} - 8 \chi_j(\eta) +1$. Put $\ell = \left \lceil \log_2(m+j) \right \rceil$. Let us first assume that $3 \cdot 2^{\ell -2} - 2^{(\ell-2)/2} < m+j \leq 2^{\ell}$. Note that 
\begin{equation} \label{l-bound}
\ell \geq \left \lceil \log_2(5 \cdot 2^{\eta -1} - 8 \chi_j(\eta) +1) \right \rceil \geq \left \lceil \log_2(2^{\eta +1}) \right \rceil = \eta + 1 = 2 \left \lceil \log_2 j \right \rceil + 21.
\end{equation}
In particular, we have that $\ell -1 \geq \left \lceil \log_2 j \right \rceil +8$. We can, therefore, apply Lemma \ref{gen3.2} to get that $\Gamma_j(\kappa_{\ell -1}) \geq 5 \cdot 2^{\ell -2} - 8 \chi_j(\ell -1) + 1$. Observe that
\begin{align}
5 \cdot 2^{\ell -2} - 8 \chi_j(\ell -1) +1 & \geq (m+j) + 2^{\ell -2} - 8 (4j +3) + 1 \\
& \geq (m+j) + \frac{1}{4} \left ( 5 \cdot 2^{\eta -1}-8 \chi_j(\eta) +1 +j \right ) - 32j -23 \\
& \geq (m+j) + \frac{1}{4} \left ( 5 \cdot 2^{2 \left \lceil \log_2 j \right \rceil +21} -8(4j+3) +j+1 \right ) - 32j -23 \\
& > m.
\end{align}
It follows that $\Gamma_j(\kappa_{\ell -1}) > m$. Because $\Gamma_j$ is nondecreasing, $\kappa_{\ell -1} > k$. Therefore, 
\begin{equation} \label{case1}
\frac{\Gamma_j(k)}{k} > \frac{3 \cdot 2^{\ell -2} - 2^{(\ell -2)/2}}{\kappa_{\ell -1}} = \frac{3 \cdot 2^{\ell -2} - 2^{(\ell -2)/2}}{2^{\ell -1}+2}
\end{equation}
in the case where $3 \cdot 2^{\ell -2} - 2^{(\ell -2)/2} < m+j \leq 2^{\ell}$.

Assume next that $2^{\ell } \leq m + j \leq 3 \cdot 2^{\ell -2} - 2^{(\ell -2)/2}$ and $\ell$ is even. By (\ref{l-bound}), we have $\ell -2 > 2 \left \lceil \log_2 j \right \rceil +18$, so
\begin{equation}
(\ell -2)/2 > \left \lceil \log_2 j \right \rceil +9 > \left \lceil \log_2 j \right \rceil +2. 
\end{equation}
We can thus apply Lemma \ref{gen3.2} to get that $\Gamma_j(k_{(\ell -2)/2 }) \geq 3 \cdot 2^{\ell -2} - 2^{(\ell -2)/2} +1 > m$. Because $\Gamma_j$ is nondecreasing, $k < k_{(\ell -2)/2}$. Thus, 
\begin{equation} \label{case2}
\frac{\Gamma_j(k)}{k} > \frac{2^{\ell -1}}{k_{(\ell -2)/2}} = \frac{2^{\ell -1}}{2^{\ell -2} + 2^{(\ell -2)/2} +2}
\end{equation}
in this case.

Finally, assume that $2^{\ell -2} \leq m+j \leq 3 \cdot 2^{\ell -2} - 2^{(\ell -2)/2}$ and $\ell$ is odd. By (\ref{l-bound}), we have $\ell -3 \geq 2 \left \lceil \log_2 j \right \rceil +18$, so
\begin{equation}
(\ell -3)/2 \geq \left \lceil \log_2 j \right \rceil +9.
\end{equation}
Therefore, Lemma \ref{gen3.2} gives that $\Gamma_j(K_{(\ell -3)/2}) \geq 3 \cdot 2^{\ell -2} - 2^{(\ell -5)/2} + 1 > m$. Since $\Gamma_j$ is nondecreasing, we have $k < K_{(\ell -3)/2}$. Consequently, 
\begin{equation} \label{case3}
\frac{\Gamma_j(k)}{k} > \frac{2^{\ell-1}}{K_{(\ell -3)/2}} = \frac{2^{\ell -1}}{2^{\ell -2} + 3 \cdot 2^{(\ell +3)/2} + 49}
\end{equation}
in this case.

By (\ref{case1}), (\ref{case2}), and (\ref{case3}), we have that in all cases,
\begin{equation}
\frac{\Gamma_j(k)}{k} > \frac{3 \cdot 2^{\ell -2} - 2^{(\ell -2)/2}}{2^{\ell -1}+2}.
\end{equation}
This gives that $\Gamma_j(k)/k$ is bounded below by a positive function of $\ell$. It follows that $\ell \rightarrow \infty$ as $k \rightarrow \infty$. Consequently, $\displaystyle \liminf_{k \rightarrow \infty} \frac{\Gamma_j(k)}{k} \geq \lim_{\ell \rightarrow \infty} \frac{3 \cdot 2^{\ell -2} - 2^{(\ell -2)/2}}{2^{\ell -1}+2} = \frac{3}{2}$.
\end{proof}

\section{Conclusion and Further Directions}

In Section \ref{sec:Gamma}, we proved that $\displaystyle \liminf_{k \rightarrow \infty} (\Gamma_j(k) / k) = 3/2$ and that $\displaystyle \limsup_{k \rightarrow \infty} (\Gamma_j(k)/k) = 3$. While we were able to prove these exact asymptotic results in Section \ref{sec:Gamma}, we were only able to obtain the asymptotic bounds $\displaystyle \frac{1}{10} \leq \liminf_{k \rightarrow \infty} \frac{\gamma_j(k)}{k} \leq \frac{9}{10}$ and $\displaystyle \frac{1}{5} \leq \limsup_{k \rightarrow \infty} \frac{\gamma_j(k)}{k} \leq \frac{3}{2}$ in Section \ref{sec:gamma}. However, as of yet, we have no reason to believe that the asymptotic behavior of $\gamma_j$ and $\Gamma_j$ depend on $j$. As such, we extend a conjecture of Defant \cite[Conjecture 22]{Defant} regarding the exact asymptotic growth of $\gamma_0$:
\begin{conj}[{cf.~\cite[Conjecture 22]{Defant}}]
For any nonnegative integer $j$, we have
\begin{equation*}
\liminf_{k \rightarrow \infty} \frac{\gamma_j(k)}{k} = \frac{9}{10} \hspace{.5cm} \text{and} \hspace{.5cm} \limsup_{k \rightarrow \infty} \frac{\gamma_j(k)}{k} = \frac{3}{2}.
\end{equation*}
\end{conj}
Note that Narayanan \cite{Narayanan} has proven $\displaystyle \limsup_{k \rightarrow \infty} ( \gamma_0(k)/k) = 3/2$.

Finally, note that it may be interesting to investigate the properties of $AP_j(x,k)$ for other infinite words $x$; Defant \cite{Defant} suggests doing this for $j = 0$. In this paper, we have utilized the recursive structure of $\t$ to prove exact asymptotic values (resp. asymptotic bounds) for $\Gamma_j(k)/k$ (resp. $\gamma_j(k)/k$) that are independent of $j$. It may be particularly interesting to know whether there are recursively defined infinite words for which the asymptotic growth of analogously defined functions depends on $j$.

\section{Acknowledgements}

This work was supported by NSF grant 1659047 and NSA grant H98230-18-1-0010 as a part of the 2018 Duluth Research Experience for Undergraduates (REU). The author would like to thank Niven Achenjang for suggesting the term \textit{$j$-fix}. The author would also like to thank Prof. Joe Gallian for organizing a wonderful REU, as well as Levent Alpoge, Aaron Berger, and Colin Defant for being outstanding advisors. Finally, the author would like to thank Gallian, Defant, and Andrew Kwon for reading through this manuscript and providing helpful feedback.  

\nocite{*}
\bibliographystyle{abbrvnat}
\bibliography{biblio}

\end{document}